\documentclass[a4paper, 11pt, reqno]{amsart}
\usepackage{amsmath,amssymb,amsthm,amsfonts,mathrsfs,color,graphicx}
\usepackage{setspace}
\usepackage{ulem}
\usepackage{stmaryrd}
\usepackage[margin=3cm]{geometry}

\usepackage[hidelinks]{hyperref}
\usepackage{cite}
\usepackage{enumerate,latexsym}

\DeclareMathOperator{\rr}{\mathbb R}

\DeclareMathOperator{\spt}{spt}

\DeclareMathOperator{\ba}{B}

\DeclareMathOperator{\s}{S}
\DeclareMathOperator{\diver}{div}

\newcommand{\R}{\mathbb{R}}

\newcommand{\si}{\Sigma}
\newcommand{\m}{{\bf M}}

\newcommand{\rt}{\mathbb{R}^{n+1}}
\newcommand{\de}{\mathcal{D}_}
\newcommand{\p}{\partial}

\newcommand{\mres}{\mathbin{\vrule height 1.6ex depth 0pt width
0.13ex\vrule height 0.13ex depth 0pt width 1.3ex}}

\newtheorem{theorem}{Theorem}
\newtheorem*{thmA}{Theorem A}
\newtheorem*{thmB}{Theorem B}

\newtheorem{lemma}{Lemma}
\newtheorem{proposition}{Proposition}

\theoremstyle{definition}\newtheorem{definition}{Definition}
\newtheorem{claim}{Claim}

\numberwithin{equation}{section}

\title[A two-piece property for free boundary minimal hypersurfaces]{A two-piece property for free boundary minimal hypersurfaces in the $(n+1)$-dimensional ball}
\author{Vanderson Lima and Ana Menezes}
\address{Instituto de Matem\'atica e Estat\'istica\\ Universidade Federal do Rio Grande do Sul\\ Brazil}
\email{vanderson.lima@ufrgs.br}

\address{Department of Mathematics\\ Princeton University\\ USA}
\email{amenezes@math.princeton.edu}

\begin{document}

\begin{abstract}
We prove that every hyperplane passing through the origin in $\rr^{n+1}$ divides an
embedded compact free boundary minimal hypersurface of the euclidean $(n+1)$-ball in
exactly two connected hypersurfaces. We also show that if a region in the $(n+1)$-ball has
mean convex boundary and contains a nullhomologous $(n-1)$-dimensional equatorial disk, then this
region is a closed halfball. Our first result gives evidence to a conjecture by Fraser
and Li in any dimension.
\end{abstract}

\maketitle

\section{Introduction}

Inspired by the work of Ros \cite{R} for closed minimal surfaces in $\s^3$, the authors proved in \cite{LM} the two-piece property for free boundary minimal surfaces in the unit ball of $\rr^3$. This result gives evidence to a conjecture by Fraser and Li \cite{FraserLi} concerning the first Steklov eigenvalue of free boundary minimal surfaces in $\ba^3.$ Also, Kusner and McGrath \cite{KM} used our result of two-piece property in the free boundary context to prove the uniqueness of the critical catenoid among embedded minimal annuli invariant under the antipodal map. This settles a case of another well-known conjecture of \cite{FraserLi} on the uniqueness of the critical catenoid.

 In the present paper we prove that the two-piece property holds in any dimension. More precisely, we prove the following.

\begin{thmA}[The two-piece property]
Every hyperplane in $\mathbb{R}^{n+1}$ passing through the origin divides an embedded compact free boundary minimal hypersurface of the unit $(n+1)$-ball $\ba^{n+1}$ in exactly two connected components.
\end{thmA}


We also prove the following result which can be seen as a strong version of the analog of the result by Solomon \cite{Sol} in the free boundary context.

\begin{thmB}
Let $W\subset \ba^{n+1}$ be a connected closed region with mean convex boundary such that $\partial W$ meets $\s^n$ orthogonally along its boundary and $\partial W$ is smooth. Suppose $W$ contains a set of the form $P\cap \ba^{n+1}$ which is nullhomologous in $W$ (see Definition \ref{def-nul}), where $P$ is a $(n-1)$-dimensional plane in $\mathbb{R}^{n+1}$ passing through the origin. Then $W$ is a closed halfball.
\end{thmB}


Let us remark that exactly as in the case $n=2$, Theorem A can be proved by assuming the conjecture by Fraser and Li \cite{FraserLi} on the the first Steklov eigenvalue of free boundary minimal hypersurfaces in $\ba^{n+1}$; hence,  Theorem A gives evidence to this conjecture (see \cite[Remark 2]{LM}).

The strategy to prove Theorem A and Theorem B is similar to the case $n=2$ and uses Geometric Measure Theory to analyze the minimizers of a partially free boundary problem for the area functional. However, in higher dimensions the situation is more delicate since the hypersurfaces obtained as minimizers can have a singular set  (see Theorem \ref{thm.reg1}).

Motivated mainly by the celebrated work of Fraser and Schoen \cite{F.S1,F.S2}, the study of free boundary minimal surfaces in $\ba^3$ saw a rapid development in the last few years, see for instance \cite{M.Li} and the references therein. However, the case of free boundary minimal hypersurfaces in $\ba^{n+1}$ is not so well-studied. Concerning examples, some free boundary minimal hypersurfaces with symmetry were constructed in \cite{FGM}, and a variational theory has been developed in \cite{LZ, SWZ, Wan}.

Regarding some properties of free boundary minimal hypersurfaces, we can mention that the asymptotic properties of the index of higher-dimensional free boundary minimal catenoids were studied in \cite{SSTZ}, and in \cite{ACS} it was proved that the index of a properly embedded free boundary minimal hypersurface in $\ba^{n+1}, 3\leq n+1\leq 7,$ grows linearly with the dimension of its first relative homology group. In \cite{Li} the first author proved the index can be controlled from above by a function of the $L^n$ norm of the second fundamental form. Also, compactness results for the space of free boundary minimal hypersurfaces were obtained in \cite{FraserLi,ACS2,GZ}.

\section{Preliminary}

\subsection{Free boundary minimal hypersurfaces}

Let $\ba^{n+1}\subset\rr^{n+1}$ be the unit ball of dimension $n+1$ with boundary $\partial\ba^{n+1}=\s^n$. Throughout this paper we will denote by $D^n$ the $n-$dimensional equatorial disk which is the intersection of $\ba^{n+1}$ with a hyperplane passing through the origin. In the following, $\mathcal{H}^s$ denotes the $s$-dimensional Hausdorff measure, where $s > 0$.

Let $\Sigma \subset \mathbb{R}^{n+1}$. Along this section we will use the following notation/assumptions: 

\begin{itemize}
\item $\overline{\si}$ is compact and it is contained in $\ba^{n+1}$.
\item $\si$ is an embedded orientable smooth hypersurface with boundary. 
\item The singular set $\mathcal{S}_{\si}$ is the complement of $\si$ in $\overline{\si}$. We suppose $\mathcal{H}^{n}(\mathcal{S}_{\si}) = 0$.
\item The boundary of $\si$ satisfies $\p\si=\Gamma_I\cup\Gamma_S$, where $\mathrm{int} (\Gamma_I)\subset \mathrm{int}(\ba^{n+1})$ and $\Gamma_S\subset \s^n.$ We have that, away from the singular set, $\p\si$ is an embedded smooth submanifold of dimension $n-1$.
\end{itemize}

\begin{definition}
Let $\Sigma$ be as above. We say that $\Sigma$ is a {\it minimal hypersurface with free boundary} if the mean curvature vector of $\Sigma$ vanishes and $\Sigma$ meets $\s^{n}$ orthogonally along $\p \si$ (in particular, $\Gamma_I=\emptyset).$ We say that $\Sigma$ is a {\it minimal hypersurface with partially free boundary} if the mean curvature vector of $\Sigma$ vanishes and its boundary $\Gamma_I\cup\Gamma_S$ satisfies that $\Gamma_I\neq\emptyset$ and $\Sigma$ meets $\s^n$ orthogonally along $\Gamma_S$.\end{definition}

From now on, given a (partially) free boundary minimal hypersurface $\Sigma\subset \ba^{n+1}$ with boundary $\p \si=\Gamma_I\cup \Gamma_S$, we will call $\Gamma_I$ its fixed boundary and $\Gamma_S$ its free boundary.

 \begin{definition}\label{def-stable}
Let $\Sigma$ be a partially free boundary minimal hypersurface in $\ba^{n+1}$. We say that $\Sigma$ is {\it stable} if for any function $f\in C^{\infty}(\Sigma)$ such that $f|_{\Gamma_I}\equiv 0$ and $\textrm{supp}(f)$ is away from the singular set $\overline{\si}\setminus\si$, we have
\begin{equation}\label{eq-stable}
-\int_{\Sigma} (f\Delta_{\Sigma}f+|A_\Sigma|^2f^2)\,d\mathcal{H}^n +\int_{\Gamma_S}\left(f\frac{\p f}{\p \nu}-f^2\right)d\mathcal{H}^{n-1} \geq 0,
\end{equation}
or equivalently
\begin{equation}\label{eq-stable2}
\int_{\Sigma} (|\nabla_\Sigma f|^2-|A_\Sigma|^2f^2)\,d\mathcal{H}^n -\int_{\Gamma_S}f^2 d\mathcal{H}^{n-1}\geq 0,
\end{equation}
where $\nu$ is the outward normal vector field to $\Gamma_S$. 
\end{definition}

Observe that if $\Sigma$ is stable then, by an approximation argument, the inequality \eqref{eq-stable2} holds for any function $f\in H^1(\Sigma)$ such that $f(p) = 0$ for a.e. $p \in \Gamma_I$ and $\textrm{supp}(f)$ is away from the singular set. In particular \eqref{eq-stable2} holds for any Lipschitz function satisfying the boundary condition.

\begin{lemma}\label{lem2}
Let $\Sigma$ be a partially free boundary minimal hypersurface in $\ba^{n+1}$ of finite area and such that the singular set $\mathcal{S}_\si=\overline{\si}\setminus \si$  satisfies $\mathcal{S}_{\si}=\mathcal{S}_0\cup\mathcal{S}_1$, where $\mathcal{S}_0 \subset\overline{\Gamma_I}$ and  $\mathcal{H}^{n-2}\big(\mathcal{S}_1\big) = 0$. If $\Gamma_I$ is contained in an n-dimensional equatorial disk, then $\Sigma$ is totally geodesic. 
\end{lemma}

\begin{proof}
Let $\Sigma$ be as in the hypotheses and denote by $D^n$ the equatorial disk that contains $\overline{\Gamma_I}$. Let $v\in \s^n$ be a vector orthogonal to the disk $D^n$ and consider the function $f(x)=\langle x,v\rangle$, $x \in \overline{\Sigma}.$ By hypothesis, we know that $f|_{\Gamma_I}\equiv 0$. A standard calculation using that $\si$ is minimal and free boundary yields 
$$\Delta_{\si}f = 0, \quad \frac{\partial f}{\partial\nu} = f.$$

Fix $\epsilon >0$ and consider a smooth function $\eta_\epsilon:[-1,1]\to [0,1]$ so that
\begin{itemize}
\item $\eta_\epsilon(s)=0$ for $|s|<\epsilon$,
\item $\eta_\epsilon(s)=1$ for $|s|>2\epsilon$,
\item $|\eta^{\prime}_\epsilon|<\displaystyle\frac{C}{\epsilon}$, for some constant $C > 0$.
\end{itemize}
Define $\phi_{0,\epsilon}:\overline{\Sigma} \to [0,1]$ as $\phi_{0,\epsilon}(x)=\eta_\epsilon(f(x))$. In particular, we have $|\nabla_\Sigma\phi_{0,\epsilon}|<C/\epsilon$ in $\Sigma.$  Observe that the set $S\subset \mathcal{S}_1$ where $\phi_{0,\epsilon}$ is not smooth satisfies $\mathcal{H}^{n-2}(S)=0.$ 

Since $\mathcal{S}_1\cap\{|f(x)|\geq\frac{\epsilon}{2}\}$ is compact and $\mathcal{H}^{n-2}\big(\mathcal{S}_1\big) = 0$, for any $\epsilon' > 0$ there exist balls $B_{r_i}(p_i) \subset \mathbb{R}^{n+1},\, i=1,\cdots,m,$ such that
$$\mathcal{S}_1\cap\left\{|f(x)|\geq\frac{\epsilon}{2}\right\} \subset \bigcup_{i=1}^{m} B_{r_i}(p_i), \quad \sum_{i=1}^{m} r_{i}^{n-2} < \epsilon',\ i=1,\cdots,m.$$
For each $i=1,\cdots,m$, consider a smooth function $\phi_{i}:\overline{\Sigma} \to [0,1]$ such that
\begin{itemize}
\item $\phi_{i}(s)=0$ in $B_{r_i}(p_i)$,
\item $\phi_{i}(s)=1$ in $\mathbb{R}^{n+1}\setminus B_{2r_i}(p_i)$,
\item $|\nabla_{\si}\phi_{i}|<\displaystyle\frac{2}{r_i},\, \forall\, x \in \si$. 
\end{itemize}

Define $\phi_{\epsilon}, f_\epsilon:\overline{\Sigma} \to [0,1]$ by $\phi_{\epsilon}(x) = \displaystyle\min_{0\leq i \leq m} \phi_{i}$, where $\phi_0=\phi_{0,\epsilon},$ and $f_\epsilon=\phi_\epsilon f$. We have that $f_{\epsilon}$ is Lipschitz  and $f_\epsilon|_{\Gamma_I}\equiv 0$, hence (\ref{eq-stable2}) holds. Moreover
$$|\nabla_\Sigma f_\epsilon|^2=\phi_\epsilon^2|\nabla_\Sigma f|^2+2f\phi_\epsilon\langle \nabla_\Sigma f, \nabla_\Sigma \phi_\epsilon\rangle + f^2|\nabla_\Sigma \phi_\epsilon|^2$$
and
$$\begin{array}{rcl}
\displaystyle\int_\Sigma \phi^2_\epsilon|\nabla_\Sigma f|^2d\mathcal{H}^n &=& \displaystyle-\int_\Sigma f\phi_{\epsilon}^2\Delta_\Sigma f d\mathcal{H}^n-\int_\Sigma 2f\phi_\epsilon\langle\nabla_\Sigma \phi_\epsilon, \nabla_\Sigma f\rangle d\mathcal{H}^n+\int_{\p \Sigma}\phi_\epsilon^2 f\frac{\p f}{\p\nu}d\mathcal{H}^{n-1}\\\\
\displaystyle&=&\displaystyle -\int_\Sigma 2f\phi_\epsilon\langle\nabla_\Sigma \phi_\epsilon, \nabla_\Sigma f\rangle d\mathcal{H}^n+\int_{\Gamma_S}\phi_\epsilon^2 f^2d\mathcal{H}^{n-1},
\end{array}$$
since $\Delta_\Sigma f\equiv 0$, $\frac{\p f}{\p \nu}=f$ and $f|_{\Gamma_I}\equiv 0.$ Hence, applying it to (\ref{eq-stable2}), we get
\begin{equation}\label{eq4}
\int_\Sigma  (f^2|\nabla_\Sigma \phi_\epsilon|^2 -|A_\Sigma|^2\phi_\epsilon^2 f^2)\,d\mathcal{H}^n \geq0.
\end{equation}

On the other hand, along $\overline{\si}$ we have $f^2 \leq 1$. Since $f$ has support away from the singular set, by the classical monotonicity formula at the interior and at the free boundary, there is $C_{\si, \epsilon} > 0$  such that
$$\mathcal{H}^n\bigl(B_{2r_i}(p_i)\cap\Sigma\bigr) \leq C_{\si,\epsilon}\,r_i^{n}.$$

Thus

\begin{align*}
\int_\Sigma  f^2|\nabla_\Sigma \phi_\epsilon|^2d\mathcal{H}^n &\leq\sum_{i=0}^{m}\int_\Sigma  f^2|\nabla_\Sigma \phi_{i,\epsilon}|^2d\mathcal{H}^n\\
&= \int_\Sigma  f^2|\nabla_\Sigma \phi_{0,\epsilon}|^2d\mathcal{H}^n + \sum_{i=1}^{m}\int_{\big(B_{2r_i}(p_i)\setminus B_{r_i}(p_i)\big)\cap\Sigma}  f^2|\nabla_\Sigma \phi_{i,\epsilon}|^2d\mathcal{H}^n\\
&\leq 4C\mathcal{H}^n\bigl(\Sigma\cap\{|f|^{-1}(\epsilon, 2\epsilon)\}\bigr) + \sum_{i=1}^m\frac{4}{r_{i}^{2}}\,\mathcal{H}^n\bigl(B_{2r_i}(p_i)\cap\Sigma\bigr)\\
&\leq 4C\mathcal{H}^n\bigl(\Sigma\cap\{|f|^{-1}(\epsilon, 2\epsilon)\}\bigr) + C'_{\si,\epsilon}\sum_{i=1}^{m} r_{i}^{n-2}\\
&\leq 4C\mathcal{H}^n\bigl(\Sigma\cap\{|f|^{-1}(\epsilon, 2\epsilon)\}\bigr) + C'_{\si, \epsilon}\,\epsilon'.
\end{align*}

If we let $\epsilon'\to0$ first and then $\epsilon\to0$ we obtain
$$
\int_\Sigma|A_\Sigma|^2 f^2d\mathcal{H}^n=0.
$$
If $|A_\Sigma|\equiv0$ then $\Sigma$ is totally geodesic and we are done. If $|A_\Sigma|(x)>0$ for some $x\in\Sigma,$ then we can find a neighborhood $U$ of $x$ in $\Sigma$ such that $|A_{\Sigma}|$ is strictly positive. This implies $\langle y, v\rangle=0$ for any $y\in U$. Therefore, $\Sigma$ is entirely contained in the disk $D^n;$ in particular, it is totally geodesic.
\end{proof}


An equatorial disk $D^{n}$ divides the ball $\ba^{n+1}$ into two (open) halfballs. We will denote these two halfballs by $\ba^+$ and $\ba^-,$ and we have $\ba^{n+1}\setminus D^n = \ba^+\cup \ba^-.$

In the next proposition we will summarize some facts about partially free boundary minimal surfaces in $\ba^{n+1}$ which we will use in the proof of Theorem \ref{thm-main}.

\begin{proposition}
\begin{enumerate}
\item[(i)] Let $D^n$ be an equatorial disk and let $\Sigma$ be a smooth partially free boundary minimal hypersurface in $\ba^{n+1}$ contained in one of the closed halfballs determined by $D^n$, say $\overline{\ba^+},$ and such that $\partial \Sigma\subset \partial \overline{\ba^+}$. If $\Sigma$ is not contained in an equatorial disk, then $\Sigma$ has necessarily nonempty fixed boundary and nonempty free boundary.

\item[(ii)] The only smooth (partially) free boundary minimal hypersurface that contains a $(n-1)$-dimensional piece of the free boundary of a $n-$dimensional equatorial disk is (contained in) this equatorial disk itself.
\end{enumerate}
\label{prop-simple}
\end{proposition}

\begin{proof}
${\it (i)}$  If the free boundary were empty, we could apply the (interior) maximum principle with the family of hyperplanes parallel to the disk $D^n$ and conclude that $\Sigma$ should be contained in the disk $D^n$. On the other hand, if the fixed boundary were empty, then we would have a minimal hypersurface entirely contained in a halfball without fixed boundary; hence, we could apply the (interior or free boundary version of) maximum principle with the family of equatorial disks that are rotations of $D^n$ around a $(n-1)-$dimensional equatorial disk and conclude that $\Sigma$ should be an equatorial disk.

${\it (ii)}$ Let $D^n$ be an equatorial disk and suppose that $\Sigma$ is a (partially) free boundary minimal hypersurface such that $\Sigma\cap D^n$ contains a $(n-1)-$dimensional piece $\Upsilon$ of the free boundary of $D^n$ in $\s^n.$  Assume, without loss of generality, $D^n\subset \{x_{n+1}=0\}$. 

Observe that since $\Sigma$ is free boundary we know that $\frac{\partial x_{n+1}}{\partial \eta}{\big |}_{\Upsilon}=x_{n+1}{\big |}_{\Upsilon}=0,$ where $\eta$ is the conormal vector to $\Upsilon;$ and since $\Sigma$ is a minimal hypersurface in $\rr^{n+1}$ we have that $x_{n+1}{\big |}_\Sigma$ is harmonic.

We will show that $x_{n+1}{\big |}_\Sigma\equiv 0.$

Consider an extension $\hat{\Sigma}$ of $\Sigma$ along $\Upsilon$ such that $\Upsilon\subset \mbox{int}(\hat{\Sigma})$ and define $\hat{x}_{n+1}$ on $\hat\Sigma$ as
$$
\Big\{ 
\begin{array}{ll}
\hat{x}_{n+1}=x_{n+1}& \ \mbox{on} \ \ \Sigma\\
\hat{x}_{n+1}=0& \ \mbox{on} \ \ \hat{\Sigma}\setminus \Sigma
\end{array}
$$

Observe that $\hat{x}_{n+1}{\big |}_\Upsilon=x_{n+1}{\big |}_\Upsilon\equiv0$, $\frac{\partial\hat{x}_{n+1}}{\partial\hat\eta}{\big |}_\Upsilon=0$ and $\frac{\partial\hat{x}_{n+1}}{\partial\eta}{\big |}_\Upsilon=\frac{\partial {x}_{n+1}}{\partial\eta}{\big |}_\Upsilon=0,$ where $\hat\eta$ is the conormal to $\Upsilon$ pointing towards $\Sigma$ and $\eta$ is the conormal to $\Upsilon$ pointing towards $\hat\Sigma\setminus\Sigma;$ hence, $\hat{x}_{n+1}$ is $C^1$ in a neighborhood of $\Upsilon$ in $\hat\Sigma.$

\begin{claim}
$\hat{x}_{n+1}$ is a weak solution to the Laplacian equation $\Delta u=0.$
\end{claim}

Observe that $\hat{x}_{n+1}$ is a harmonic function on $\hat\Sigma\setminus\Upsilon$, so we just need to show the claim in a neighborhood of $\Upsilon.$

Consider a domain $\Omega=\Omega_1\cup\Omega_2$ where $\partial\overline{\Omega_i}=\Gamma_i\cup(\overline\Omega\cap\Upsilon)$ with $\Omega_1\subset\hat\Sigma\setminus\Sigma$ and $\Omega_2\subset\Sigma$ (see Figure \ref{fig_domain1}), and let $\phi:\hat\Sigma\to\rr$ be a smooth function with compact support contained in $\Omega.$ 

\begin{figure}[!h]
\includegraphics[scale=0.9]{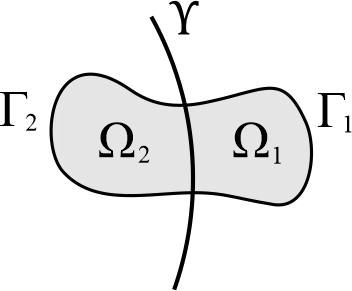}
\caption{$\Omega=\Omega_1\cup\Omega_2$.}
\label{fig_domain1}
\end{figure}

Integration by parts gives us
$$\begin{array}{rcl}
\displaystyle\int_\Omega\langle\nabla\phi, \nabla\hat{x}_{n+1}\rangle {\rm d}\sigma&=&\displaystyle-\int_\Omega\hat{x}_{n+1}\Delta\phi \, {\rm d}\sigma+\int_{\partial\Omega}\hat{x}_{n+1}\langle\nabla\phi, \nu\rangle{\rm d}L\\
&=&\displaystyle-\int_\Omega\hat{x}_{n+1}\Delta\phi \, {\rm d}\sigma,
\end{array}$$
since supp$(\phi)\subset \subset \Omega,$ where $\nu$ is the outward conormal to $\partial\Omega.$

Then,
$$\begin{array}{rcl}
-\displaystyle\int_\Omega\hat{x}_{n+1}\Delta\phi \,{\rm d}\sigma&=&\displaystyle\int_\Omega\langle\nabla\phi, \nabla\hat{x}_{n+1}\rangle {\rm d}\sigma\\
&=&\displaystyle\int_{\Omega_1}\langle\nabla\phi, \nabla\hat{x}_{n+1}\rangle {\rm d}\sigma+\displaystyle\int_{\Omega_2}\langle\nabla\phi, \nabla\hat{x}_{n+1}\rangle {\rm d}\sigma.\\
\end{array}$$

We have
$$\begin{array}{rcl}
\displaystyle\int_{\Omega_1}\langle\nabla\phi, \nabla\hat{x}_{n+1}\rangle {\rm d}\sigma&=& -\displaystyle\int_{\Omega_1}\phi\Delta\hat{x}_{n+1}{\rm d}\sigma+\displaystyle\int_{\partial\Omega_1}\phi\langle\nabla \hat{x}_{n+1}, \nu_1\rangle{\rm d}L\\
&=&\displaystyle\int_{\Upsilon}\phi\langle\nabla \hat{x}_{n+1}, \nu_1\rangle{\rm d}L\\
&=&0,
\end{array}$$
where in the first equality we used that $\hat{x}_{n+1}{\big |}_{\Omega_1}\equiv 0$ and in the second equality we used the fact that $\frac{\partial\hat{x}_{n+1}}{\partial\nu_1}{\big |}_\Upsilon=0,$ where $\nu_1$ is the outward conomal to $\Upsilon$ with respect to $\Omega_1.$

Analogously, we have
$$\begin{array}{rcl}
\displaystyle\int_{\Omega_2}\langle\nabla\phi, \nabla\hat{x}_{n+1}\rangle {\rm d}\sigma&=& -\displaystyle\int_{\Omega_2}\phi\Delta\hat{x}_{n+1}{\rm d}\sigma+\displaystyle\int_{\partial\Omega_2}\phi\langle\nabla \hat{x}_{n+1}, \nu_1\rangle{\rm d}L\\
&=&\displaystyle\int_{\Upsilon}\phi\langle\nabla \hat{x}_{n+1}, \nu_2\rangle{\rm d}L\\
&=&0,
\end{array}$$
where in the first equality we used that $\hat{x}_{n+1}{\big |}_{\Omega_2}=x_{n+1}{\big |}_{\Omega_2}$ is harmonic and in the second equality we used the fact that $\frac{\partial\hat{x}_{n+1}}{\partial\nu_2}{\big |}_\Upsilon=0,$ where $\nu_2$ is the outward conomal to $\Upsilon$ with respect to $\Omega_2.$

Therefore, the claim follows and, by the Elliptic theory, $\hat{x}_{n+1}$ has to be a (strong) solution to the Laplacian equation. Moreover, since $\hat{x}_{n+1}$ vanishes on an open set, the unique continuation result implies that $\hat{x}_{n+1}\equiv 0$ on $\hat\Sigma,$ that is, $\Sigma$ is (contained in) the equatorial disk $D^n.$
\end{proof}

\subsection{Integer rectifiable varifolds}

A set $M \subset \R^{n+1}$ is called countably $k$-rectifiable if $M$ is $\mathcal{H}^k$-measurable and if
$$M \subset \bigcup_{j = 0}^{\infty} M_{j},$$
where $\mathcal{H}^k(M_0) = 0$ and for $j \geq 1$, $M_j$ is an $k$-dimensional $C^1$-submanifold of $\R^{n+1}$. Such $M$ possesses $\mathcal{H}^k$-a.e. an approximate tangent space $T_{x}M$.

Let $G(n+1,k)$ be the Grassmannian of $k$-hyperplanes in $\mathbb{R}^{n + 1}$. An integer multiplicity rectifiable $k$-varifold $\mathcal{V} = v(M,\theta)$ is a Radon measure on $U\times G(n+1,k)$, defined by
$$\mathcal{V}(f) = \int_{M}f(x,T_x M)\,\theta(x)\,d\mathcal{H}^k, \ f \in C_{0}^{c}\big(U\times G(n+1,k)\big),$$
where $M \subset U$ is countably $k$-rectifiable and $\theta > 0$ is a locally $\mathcal{H}^k$-integrable integer valued function. Also, we say $\mathcal{V} = v(M,\theta)$ is {\it stationary} if
\begin{equation}\label{FVF}
\int_{M}\big(\diver_{M} \zeta\big)\theta\,d\mathcal{H}^k = 0,
\end{equation} 
for any $C^1$-vector field $\zeta$ of compact support. 

Then we have the following result, see \cite{I}.

\begin{lemma}\label{lem:char.stat}
Let $\si \subset \mathbb{R}^{n+1}$ be an embedded $C^1$-hypersurface such that $\mathcal{H}^{n-1}\big((\overline{\si}\setminus\si)\cap U\big) = 0$, for every open set $U \subset \mathbb{R}^{n+1}$ with compact closure. Let $\theta > 0$ be a integer valued function which is locally constant. Then, the following conditions are equivalent:
\begin{enumerate}
\item $\mathcal{V} = v(\si,\theta)$ is stationary.
\item $\vec{H}_{\si} = 0$, and there is $C_{\si} > 0$ such that for any ball $B_r(p) \subset \mathbb{R}^{n+1}$ we have 
$$\mathcal{H}^n\bigl(B_{r}(p)\cap\Sigma\bigr) \leq C_{\si}\,r^{n}.$$
\end{enumerate}
\end{lemma}

\subsection{Minimizing Currents with Partially Free Boundary}\label{pfb}

In this section we will use the following notation.

\begin{itemize}
\item $U \subset \R^{n + 1}$ is an open set;
\item $\mathcal{D}^k(U) = \{C^{\infty}\textrm{-} \ k\textrm{-}\textrm{forms} \ \omega ;\ \spt \ \omega \subset U\}$;
\item $\de k(U)$ denotes the dual of $\mathcal{D}^k(U)$, and its elements are called $k$-currents with support in $U$;
\item The mass of $T \in \de k(U)$ in $W$ is defined by
$$\m_W(T) := \sup \{T(\omega); \ \omega \in \mathcal{D}^k(U), \ \spt \omega \subset W, \ |\omega| \leq 1\} \leq +\infty;$$
\item The boundary of $T \in \de k(U)$ is the $(k - 1)$-current $\p T \in \de {k-1}(U)$ given by 
$$\p T(\omega) := T(d\omega),$$
where $d$ denotes the exterior derivative operator.
\end{itemize}

Consider a compact domain $W \subset \rr^{n+1}$ such that $\partial W = S\cup M$, where $S$ is a compact $C^2-$hypersurface (not necessarily connected) with boundary, $M$ is a smooth compact mean convex hypersurface with boundary, which intersects $S$ orthogonally along $\partial S$, and such that $\mathrm{int}(S)\cap\mathrm{int}(M) = \emptyset$. 

Let $\Omega\subset W$ be a compact hypersurface with boundary $\Gamma = \partial \Omega$. We assume that $\Gamma \cap \mathrm{int}(W)$ is an embedded $C^2$-submanifold of dimension $n-1$ away from a singular set $\mathcal{S}_0$ such that $\mathcal{H}^{n-1}(\mathcal{S}_0) = 0$.

Define the class $\mathfrak{C}$ of admissible currents by
\begin{eqnarray*}
\mathfrak{C} = \{T \in \de n(\rr^{n+1}); \ T \ \mbox{is integer multiplicity rectifiable},\\
\spt T \subset W \ \mbox{and is compact}, \ \mbox{and} \ \spt \bigl(\llbracket\Gamma\rrbracket - \p T\bigr) \subset S\},
\end{eqnarray*}
where $\llbracket\Gamma\rrbracket$ is the current associated to $\Gamma$ with multiplicity one. We want to minimize area in $\mathfrak{C}$, that is, we are looking
for $T \in \mathfrak{C}$ such that
\begin{equation}\label{var.prob}
\m(T) = \inf\{\m(\tilde{T}); \ \tilde{T} \in \mathfrak{C}\}.
\end{equation}

Observe that $\mathfrak{C} \neq \emptyset$ since $\llbracket\Omega\rrbracket\in \mathfrak{C}$. Hence, it follows from ~\cite[$5.1.6(1)$]{Fed}, that the variational problem \eqref{var.prob} has a solution (see also \cite{Gr1}). If $T \in \mathfrak{C}$ is a solution we have
\begin{eqnarray}
\m(T) &\leq & \m(T + X) \label{min.property},\\
\spt T &\subset & W,\\
\mu_T (S) &=& 0,
\end{eqnarray}
for any integer multiplicity current $X \in \de n(\rt)$ with compact support such that $\spt X \subset W$ and $\spt \p X \subset S$.

In order to apply the known regularity theory for $T$ we need the following result, whose proof is the same as that of the case $n=2$ (see Section 3 in \cite{LM}).

\begin{proposition}\label{max.princ}
If $T$ is a solution of \eqref{var.prob}, then either $\spt T\setminus\Gamma \subset W \setminus M$ or $\spt T \subset M$.
\end{proposition}


For any given $n-$dimensional compact set $K\subset W$ we call {\it corners} the set of points of $\partial K$ which also belong to $S$. We then have the following regularity result. 

\begin{theorem}\label{thm.reg1}
Let $T$ be a solution of \eqref{var.prob}.  Then there is a set $\mathcal{S}_1 \subset \spt T$ such that, away from $\mathcal{S}_0\cup\mathcal{S}_1\cup\Gamma$, $T$ is supported in a oriented embedded minimal $C^2$-hypersurface, which meets $S$ orthogonally along $\spt \bigl(\llbracket\Gamma\rrbracket - \p T\bigr)$. Moreover
\begin{equation}\label{sing.set}
\left\{\begin{array}{rl} 
&\mathcal{S}_1 = \emptyset,\quad  \text{if}\ n \leq 6 ,\\\\ 
&\mathcal{S}_1 \ \textrm{is discrete},\quad  \text{if}\ n = 7,\\\\
&\mathcal{H}^{n-7+\delta}\big(\mathcal{S}_1\big) = 0,\ \forall\, \delta > 0,  \quad  \text{if}\ n > 7.
\end{array} \right.
\end{equation}
\end{theorem}

\begin{proof}
 Let us write 
$$\spt T\setminus (\mathcal{S}_0) = \mathcal{R} \cup \mathcal{S}_{1},$$
where the union is disjoint and $\mathcal{R}$ consists of the points $x\in \spt T$ such that there is a neighborhood $U \subset \mathbb{R}^{n+1}$ of $x$ where $T\mres U$ is given by $m$-times ($m \in \mathbb{N}$) integration over an embedded $C^2$-hypersurface with boundary. To complete the proof we will prove the following:

\begin{itemize}
\item {\it Regularity at the interior}: $(\spt T\setminus \spt \partial T)\cap\mathcal{R} \neq \emptyset$ and $(\spt T\setminus \spt \partial T)\cap\mathcal{S}_{1}$ satisfies \eqref{sing.set};
\item {\it Regularity at the free-boundary}: $\spt \bigl(\llbracket\Gamma\rrbracket - \p T\bigr)\cap\mathcal{R} \neq \emptyset$ and $\spt \bigl(\llbracket\Gamma\rrbracket - \p T\bigr)\cap\mathcal{S}_{1}$ satisfies \eqref{sing.set};
\end{itemize}

The interior regularity is a classical result, see \cite[Section 5.3]{Fed}. Since by Proposition \ref{max.princ} the free part of the boundary is contained in $S\setminus \p S$, we can use the result by Gr\"uter \cite{Gr2} to conclude the regularity at the free boundary (away from the corners).
\end{proof}

\section{The two-piece property and other results}

\begin{definition}
Let $W$ be a region in $\ba^{n+1}$ and let $\Upsilon\subset W$ be a $(n-1)-$dimensional equatorial disk (that is, the intersection of $\ba^{n+1}$ with an $(n-1)$-dimensional plane passing through the origin). We say that $\Upsilon$ is {\it nullhomologous} in $W$ if there exists a compact hypersurface $M \subset W$ such that $\partial M = \Upsilon\cup\Gamma$, where $\Gamma$ is a $(n-1)-$dimensional compact set contained in $\s^n$ (see Figure \ref{fig-null}).
\label{def-nul}
\end{definition}

\begin{figure}[!h]
\centering
\includegraphics[scale=0.6]{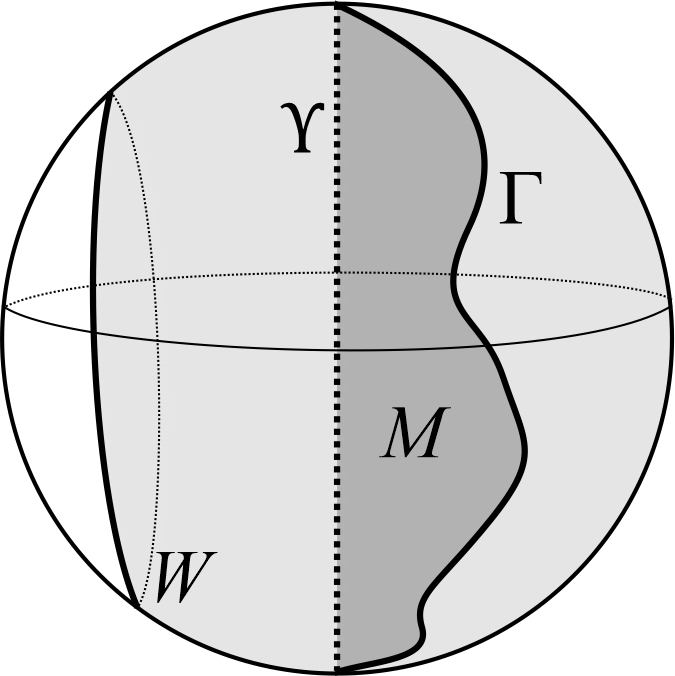}
\caption{In this region $W$, any $(n-1)-$dimensional equatorial disk $\Upsilon\subset W$ is nullhomologous.}
\label{fig-null}
\end{figure}

The boundary of the region $W$ can be written as $U\cup V$, where $\mathrm{int} (U)\subset \mathrm{int} (\ba^{n+1})$ and $V\subset \s^n.$ In the next theorem we will denote by $\p W$ the closure of the component $U,$ that is, $\p W=\overline U.$ 

\begin{theorem}
Let $W\subset \ba^{n+1}$ be a connected closed region with (not necessarily strictly) mean convex boundary such that $\partial W$ meets $\s^n$ orthogonally along its boundary and $\partial W$ is smooth. If $W$ contains a $(n-1)-$dimensional equatorial disk $\Upsilon$, and $\Upsilon$ is nullhomologous in $W$, then $W$ is a closed $(n+1)$-dimensional halfball.
\label{thm-nul}
\end{theorem}

\begin{proof}
Up to a rotation of $\Upsilon$ around the origin, we can assume that $\Upsilon\cap \p W$ is nonempty. Since $\Upsilon$ is nullhomologous in $W$, there exists a compact hypersurface $M$ contained in $W$ such that $\partial M=\Upsilon\cup\Gamma$, where $\Gamma$ is a $(n-1)-$dimensional compact set contained in $\s^n.$ We consider the class of admissible currents
\begin{eqnarray*}
\mathfrak{C} = \{T \in \de n(\rt); \ T \ \mbox{is integer multiplicity rectifiable},\\
\spt T \subset W \ \mbox{and is compact}, \ \mbox{and} \ \spt \bigl(\llbracket\partial M\rrbracket - \p T\bigr) \subset \s^n\cap W\},
\end{eqnarray*}
where $\llbracket\partial M\rrbracket$ is the current associated to $\partial M$ with multiplicity one, and we minimize area (mass) in $\mathfrak{C}.$ Then, by the results presented in Section \ref{pfb}, we get a compact embedded (orientable) partially free boundary minimal hypersurface $\Sigma\subset W$ which minimizes area among compact hypersurfaces in $W$ with boundary on the class $\Gamma=\Upsilon\cup\tilde\Gamma;$ in particular, its fixed boundary is exactly $\Upsilon.$ Moreover, by Proposition \ref{max.princ} in Section \ref{pfb}, either $\Sigma\subset\p W$ or $\Sigma\cap \p W\subset \Upsilon.$

Now the same arguments we used in the proof of Theorem 1 in \cite{LM} can be applied. In fact:
\begin{claim}
\label{claim1}
$\Sigma$ is stable.
\end{claim}

In the case $\partial W\cap\si \subset \Gamma$, $\si$ is automatically stable in the sense of Definition \ref{def-stable}, since it minimizes area for all local deformations.

Suppose $\si \subset \p W$. For any $f\in C^\infty_{c}(\Sigma)$ with $f|_{\Upsilon}\equiv0$, consider $Q(f,f)$ defined by
$$Q(f,f)=\frac{\int_\Sigma\left(|\nabla_\Sigma f|^2-|A_\Sigma|^2f^2\right)d\mathcal{H}^{n} -\int_{\Gamma}f^2d\mathcal{H}^{n-1}}{\int_{\si}f^2 d\mathcal{H}^{n}},$$ 
and let $f_1$ be a first eigenfunction, i.e., $Q(f_1,f_1) = \inf_{f} Q(f,f)$.

Observe that although differently from the classical stability quotient (we have an extra term that depends on the boundary of $\Sigma$) we can still guarantee the existence of a first eigenfunction. In fact, since for any $\delta>0$ there exists $C_\delta>0$ such that $||f||_{L^2(\partial\Sigma)} \leq \delta ||\nabla f||_{L^2(\Sigma)}+ C_\delta ||f||_{L^2(\Sigma)}$, for any $f\in W^{1,2}(\Sigma)$, we can use this inequality to prove that the infimum is finite. Once this is established the classical arguments to show the existence of a first eigenfunction work.

Since $|\nabla |f_1||=|\nabla f_1|$ a.e., we have $Q(f_1,f_1)=Q(|f_1|,|f_1|)$, that is, $|f_1|$ is also a first eigenfunction. Since $|f_1|\geq 0$, the maximum principle implies that $|f_1|>0$ in $\Sigma\setminus\partial\si$, in particular, $f_1$ does not change sign in $\Sigma\setminus\partial\si$. Then we can assume that $f_1>0$ in $\Sigma\setminus\partial\si$ and, by continuity, we get $f_1\geq0$ in $\Gamma.$ Therefore, we can use $f_1$ as a test function to our variational problem:
 Let $\zeta$ be a smooth vector field such that $\zeta(x) \in T_x\s^n,$ for all $ x \in \s^n$, $\zeta(x)\in (T_x\si)^\perp,$ for all $x \in \si$, and $\zeta$ points towards $W$ along $\si$. Let $\Phi$ be the flow of $\zeta$. For $\varepsilon$ small enough the hypersurfaces $\si_t = \{\Phi\bigl(x,tf_1\bigr)$; $x \in \si$, $0 < t < \varepsilon\}$ are contained in $W$. Since $\Sigma$ has least area among the hypersurfaces $\si_t$, we know that 
$$0 \leq \frac{d^2}{dt^2}\biggl|_{t=0^+}|\si_t| = \int_{\Sigma} (|\nabla_\Sigma f_1|^2-|A_\Sigma|^2f_1^2)\,d\mathcal{H}^{n} -\int_{\Gamma}f_1^2d\mathcal{H}^{n-1},$$
which implies that $Q(f_1,f_1)\geq0$. Since $f_1$ is a first eigenfunction, we get that $Q(f,f)\geq0$ for any $f\in C_c^\infty(\Sigma)$ with $f|_{\Upsilon}\equiv0$. Therefore, we have stability for $\Sigma$.

Then, since $\Upsilon$ is contained in an equatorial disk $D^n$, Lemma \ref{lem2} implies that $\Sigma$ is necessarily a half $n-$dimensional equatorial disk. If  $\Sigma\subset\p W$, then we already conclude that $W$ has to be a $(n+1)-$dimensional halfball.

Suppose $\Sigma\cap \p W\subset \Upsilon$. Rotate $\Sigma$ around $\Upsilon$ until the last time it remains in $W$ (this last time exists once $\Sigma\cap \p W$ is nonempty), and let us still denote this rotated hypersurface by $\Sigma$. In particular, there exists a point $p$ where $\Sigma$ and $\p W$ are tangent. We will conclude that $W$ is necessarily a $(n+1)-$dimensional halfball.

In fact, if $p\in \mbox{int}(\Upsilon),$ we can write $\p W$ locally as a graph over $\si$ around $p$ and apply the classical Hopf Lemma; if $p\in \p \Upsilon,$ we can use the Serrin's Maximum Principle at a corner (see Appendix A in \cite{LM} for the details); and if $p\in \Sigma\setminus\Upsilon$ we can apply (the interior or the free boundary version of) the maximum principle. In any case, we get that $W$ is a $(n+1)-$dimensional halfball.
\end{proof}

Now we prove the two-piece property for free boundary minimal hypersurfaces in $\ba^{n+1}$.

\begin{theorem}
Let $M$ be a compact embedded smooth free boundary minimal hypersurface in $\ba^{n+1}.$ Then for any equatorial disk $D^n$, $M\cap \ba^+$ and $M\cap \ba^-$ are connected.
\label{thm-main}
\end{theorem}

\begin{proof}
If $M$ is an equatorial disk, then the result is trivial. So let us assume this is not the case.

Suppose that, for some equatorial disk $D^n$, $M\cap \ba^+$ is a disjoint union of two nonempty open hypersurfaces $M_1$ and $M_2$, $M_1$ being connected. Notice that by Proposition \ref{prop-simple}(i) both $\overline{M_1}$ and (all components of) $\overline{M_2}$ have nonempty fixed boundary and nonempty free boundary. Let us denote by $\Gamma_I = \partial \overline{M_1}\cap D^n$ the fixed boundary of $\overline{M_1}$ which might be disconnected. If $M_1$ and $D^n$ are transverse, then $\Gamma_I$ is an embedded smooth submanifold of dimension $n-1$. If $M_1$ and $D^n$ are tangent, then the local description of nodal sets of elliptic PDE's (see for instance \cite{HarSim2}) imply that $\mathrm{int}(\Gamma_I)$ is an embedded smooth submanifold of dimension $n-1$, away from a singular set $\mathcal{S}_0$ such that $\mathcal{H}^{n-1}(\mathcal{S}_0) = 0$.


Denote by $W$ and $W'$ the closures of the two components of $\ba^{n+1}\setminus M.$ They are compact domains with mean convex boundary. Hence, we can minimize area for the following partially free boundary problem (see Section \ref{pfb}):

We consider the class of admissible currents
\begin{eqnarray*}
\mathfrak{C} = \{T \in \de n(\rt); \ T \ \mbox{is integer multiplicity rectifiable},\\
\spt T \subset W \ \mbox{and is compact}, \ \mbox{and} \ \spt \bigl(\llbracket\partial \overline{M_1}\rrbracket - \p T\bigr) \subset \s^2\cap W\},
\end{eqnarray*}
where $\llbracket\partial \overline{M_1}\rrbracket$ is the current associated to $\partial \overline{M_1}$ with multiplicity one, and we minimize area (mass) in $\mathfrak{C}.$ Then, by the results presented in Section \ref{pfb}, we get a compact embedded (orientable) partially free boundary minimal hypersurface $\Sigma\subset W$ which minimizes area among compact hypersurfaces in $W$ with the same fixed boundary as $\overline{M_1}$, which is contained in $D^n.$ Moreover, by Proposition \ref{max.princ} in Section \ref{pfb}, either $\Sigma\subset\p W$ or $\Sigma\cap \p W\subset \partial\Sigma.$

Arguing as in Claim \ref{claim1} of Theorem \ref{thm-nul}, we can prove the stability of $\Sigma$. Also, observe that by Theorem \ref{thm.reg1} the singular set $\mathcal{S}_1$ of $\si\setminus D$ is empty or satisfies $\mathcal{H}^{n-7+\delta}(\mathcal{S}_1)=0, \,\forall\, \delta > 0$, in particular $\mathcal{H}^{n-2}(\mathcal{S}_1)=0$. So, we can apply Lemma \ref{lem2} and conclude that each component of $\Sigma$ is a piece of an equatorial disk.

The case $\si \subset \p W$ can not happen because this would imply that $M$ is a disk, and we are assuming it is not. Therefore, only the second case can happen, that is, any component of $\Sigma$ meets $\p W$ only at points of $\partial\Sigma.$ Observe that each component of $\Sigma$ that is not bounded by a $(n-1)-$dimensional equatorial disk is necessarily contained in $D^n$. If some component of $\Sigma$ were bounded by a $(n-1)-$dimensional equatorial disk, then we could apply Theorem \ref{thm-nul} and would conclude that $M$ is a $n-$dimensional equatorial disk, which is not the case. Then $\Sigma$ is entirely contained in $D^n$ and, since $\Sigma\cap \p W\subset \partial\Sigma$, $M\subset \p W$ and $M\cap D^n$ does not contain any $(n-1)$-dimensional piece of $\partial D^n$ (Proposition \ref{prop-simple}(ii)), we have $\Sigma\cap M=\Gamma_I$.

Doing the same procedure as in the last paragraph for $W'$, we can construct another compact hypersurface $\Sigma'$ of $D^n$ with fixed boundary $\p \Sigma'=\Gamma_I$ and such that $\Sigma'\subset W'$ and $\Sigma'\cap M=\Gamma_I$. Notice that $\Sigma\cup \Sigma'$ is a hypersurface without fixed boundary of $D^n$, therefore $\Sigma\cup \Sigma'=D^n.$ In fact, let us denote by $T$ and $T'$ the minimizing currents associated to $\Sigma$ and $\Sigma' $ respectively, that is, $\spt T=\Sigma$ and $\spt T'=\Sigma'$. First observe that $\spt \partial (T-T')\subset \partial D^n$ and $\partial\partial (T-T')=0$; hence, by the Constancy Theorem, we know that $\partial (T-T')=k \partial D^n$, for some interger $k$.  Now, since $\spt (T-T'-kD^n)\subset D^n$ and $\partial(T-T'-kD^n)=0,$ the Constancy Theorem implies that $T-T'=kD^n$; but since $\Gamma_I$ has multiplicity one, this also holds for $T$ and $T'$ and therefore $k=1$ necessarily. Hence, $\Sigma\cup\Sigma'=\spt (T-T')=D^n$.

In particular, $M\cap D^n=\Gamma_I,$ which implies that $M_2=M\cap \ba^+\setminus M_1$ has fixed boundary contained in $\Gamma_I$. For $n=2$, since $M$ is embedded and $\Gamma_I$ has singularities of $n$-prong type (if any), we know that the fixed boundaries of $M_1$ and $M_2$ are disjoint, in particular, the fixed boundary of $\overline{M_2}$ is necessarily empty and this yields a contradiction by Proposition \ref{prop-simple}(i). It remains to analyse the case when $n\geq 3.$ 

Let us assume, without loss of generality, that $D^n= \ba\cap \{x_{n+1}=0\}$; hence, we have $M\cap D^n=\Gamma_I=\{q\in M; x_{n+1}(q)=0\}$ which is the nodal set of the Steklov eigenfunction $x_{n+1}:M\to \rr.$ 

Observe that if $q\in \Gamma_I$ and $\nabla_M x_{n+1}(q)\neq 0$ then, since $M$ is embedded, we know that in a neighborhood of $q$ we have $M\cap D^n=\overline{M_1}\cap D^n$; in particular, $q$ can not be contained in $\partial \overline{M_2}.$ 

Now let us analyse the singular set $\mathcal S=\{x_{n+1}=0\}\cap \{\nabla_M x_{n+1}=0\}\subset \Gamma_I.$ By Theorem 1.7 in \cite{HarSim2}, the Hausdorff  dimension of $\mathcal S$ is less than or equal to $n-2;$ in particular, $\mathcal H^{n-1}(\mathcal S)=0$ and therefore by Lemma \ref{lem:char.stat} $\overline{M_2}$ is stationary. By \cite{SoWhi} we can conclude that either $\overline{M_2}\cap D^n=\emptyset$ or $D^n\subset \overline{M_2}$ (which we already know is not possible). Therefore, $\overline{M_2}$ has empty fixed boundary which is a contradiction by Proposition \ref{prop-simple}(i).

Therefore, the theorem is proved.
\end{proof}

\end{document}